\documentclass[12pt,twoside]{article}

\usepackage[ansinew]{inputenc}
\usepackage{amsfonts}
\usepackage{latexsym}
\usepackage{amsmath}
\usepackage{amssymb}

\newtheorem{lettertheorem}{Theorem}

\newtheorem{defin}{Definition}

\newtheorem{theorem}{Theorem}

\newtheorem{exa}{Example}

\newtheorem{lemma}[defin]{Lemma}

\newenvironment{proof}
{\noindent{\it Proof.}}{\hfill $\Box$\par\vspace{2.5mm}}
\newenvironment{proof2}
{\noindent{\it Proof of Theorem 1.}}{\hfill $\Box$\par\vspace{2.5mm}}
\newenvironment{remark}
{\par\vspace{2.5mm}\noindent{\bf Remark.}}{\par\vspace{2.5mm}}

\newtheorem{que}{Question}

\newtheorem{pro}{Problem}

\numberwithin{equation}{section}

\makeatletter
\renewcommand{\ps@myheadings}{%
\renewcommand{\@evenfoot}{}%
\renewcommand{\@oddfoot}{}%
}\makeatother 

\title{Admissible solutions of delay Schwarzian differential equations }

\author{Shi-Jian Wu}
\date{}

\begin{document}
\pagestyle{plain}

\maketitle

\begin{abstract}
  In this paper, we study delay differential equations involving the Schwarzian derivative $S(f,z)$, expressed in the form 
  \begin{equation*}
  f(z+1)f(z-1) + a(z)S(f,z) =R(z,f(z))= \frac{P(z,f(z))}{Q(z,f(z))}
  \end{equation*}
  where $a(z)$ is rational, $P(z,f)$ and $Q(z,f)$ are coprime polynomials in $f$ with rational coefficients.  
  Our main result shows that if a subnormal transcendental meromorphic solution exists, then the rational function  $R(z,f)=P(z,f)/Q(z,f)$ satisfies $\deg_fR\leq 7$ and $\deg_fP\leq \deg_fQ +2$, where $\deg_fR  =\max\{\deg_fP, \deg_fQ\}.$ 
  Furthermore, for any rational root $b_1$ of $Q(z,f)$ in $f$ with multiplicity $k$, we show that $k \leq 2$.   
 Finally, a classification of such equations is provided according to the multiplicity structure of the roots of $Q(z,f)$. Some examples are given to support these results. 
  \medskip
  \noindent
 
  
  \medskip
  \noindent
  \textbf{2020MSC}: 34M04; 30D35
  
  \end{abstract}

\section{Introduction}

An ordinary differential equation is said to possess the Painlev$\acute{e} $ property if its solutions are single-valued about all movable singularities\cite{ConteBook,  LaineBook2}. 
Recent interest in this property stems from statistical physics and partial differential equations; for instance, these equations with Painlev$\acute{e} $ property provide exact solutions for the two-dimensional Ising model \cite{ConteBook1}. 
Beyond these applications, the significance of studying equations with Painlev$\acute{e} $ property lies in their dual role: they serve both as a source for defining new functions and a class of equations to be integrated with the existing functions available. 
This classification originated in the early 20th century through the work of Painlev$\acute{e}$\cite{Pain-1900,Pain-1902}, Fuchs\cite{Fuchs-1907} and Gambier\cite{Gamb-1907}, who systematically classified second order differential equations with the Painlev$\acute{e} $ property. 
Their analysis culminated in six canonical forms, now named as the Painlev$\acute{e} $ equations. 

The analogues of Painlev$\acute{e} $ property for complex difference equations have been discussed.  
Ablowitz, Halburd and Herbst\cite{Ablo-2000} have advocated that the existence of sufficiently many finite order meromorphic solutions could be considered as a version of the Painlev$\acute{e} $ property for difference equations. 
Their work established Nevanlinna theory as a fundamental tool for studying complex difference equations. 
Building on this foundation,  Halburd and Korhonen\cite{Halb-2007} proved that if the difference equation 
\begin{equation}\label{eqi6}
  f(z+1)+f(z-1)=R(z,f), 
\end{equation}  
where $R(z,f)$ is rational in $f$ with meromorphic coefficients, has an admissible meromorphic solution of finite order, then it reduces to a short list of canonical forms including the difference Painlev$\acute{e} $ ${\mathrm{I}}$ and ${\mathrm{II}}$  equations. 
Further studies of difference Painlev$\acute{e} $ equations have been carried out by Halburd and Korhonen \cite{Halb2-2007}, Ronkainen \cite{Ronk-2010}, and Wen \cite{Wen-2016, Wen2-2016}.  

Complex equations combining difference operators and derivatives of meromorphic functions are termed complex differential-difference equations or complex delay differential equations \cite{LKBook}. 
Some reductions of integrable differential-difference equations are known to yield delay differential equations with formal continuum limits to Painlev$\acute{e} $ equations.
For instance,  Quispel, Capel and Sahadevan\cite{Quis-1992} derived 
\begin{equation}\label{eqi1}
  f(z)[f(z+1)-f(z-1)]+af'(z)=bf(z),
\end{equation}
where $a$ and $b$ are constants, through symmetry reduction of the Kac-van Moerbeke equation. This equation possesses a formal continuum limit to the first Painlev$\acute{e} $ equation:  
\begin{equation*}
  \frac{d^2 y}{dt^2}=6y^2+t.
\end{equation*}

Subsequently, Halburd and Korhonen \cite{Halb-2017} investigated a generalization of \eqref{eqi1} and reduced this extended equation. Fundamental concepts of Nevanlinna theory are detailed in \cite{LaineBook}. Their main result is given below\cite[Theorem 1.1]{Halb-2017}: 
 \begin{lettertheorem}
Let $f(z)$ be a transcendental meromorphic solution of 
\begin{equation}\label{eqi2}
  f(z+1)-f(z-1)+a(z)\frac{f'(z)}{f(z)}=R(z,f(z))=\frac{P(z,f(z))}{Q(z,f(z))},
\end{equation}
where $a(z)$ is rational, $P(z,f(z))$ is a polynomial in $f$ having rational coefficients in $z$, and $Q(z,f(z))$ is a polynomial in $f$ with roots that are non-zero rational functions of $z$ 
and not roots of $P(z,f(z))$. If the hyper order of $f(z)$ is less than one, then 
\begin{equation*}
  \deg_fP=\deg_fQ+1\leq 3 \ or \ \deg_fR : =\max\{\deg_fP, \deg_fQ\}\leq 1.
\end{equation*}
\end{lettertheorem}

Recent extensions of this theorem have been explored in \cite{Cao-2023,Cao-2022,Hu-2021,Nie-2025,Hu-2019}.  
Notably, Nie, Huang, Wang and Wu\cite{Nie-2025} replaced the logarithmic derivative in \eqref{eqi2} with the Schwarzian derivative 
\begin{equation*}
  S(f,z) := \frac{f'''(z)}{f'(z)} - \frac{3}{2} \left( \frac{f''(z)}{f'(z)} \right)^2,
  \end{equation*}
and analyzed reductions of the resulting equation. 
This work was motivated by Malmquist's results\cite{Malm-1913} on the equation $f'=R(z,f)$ and Ishizaki's classification \cite{Ishi-1991} of the Schwarzian differential equation $S(f,z)^n = R(z,f)$ 
into six canonical forms (up to certain transformations) for positive integers $n$. Before stating the main theorem, we recall the following definition from Nevanlinna theory: a transcendental meromorphic function 
$f$ is said to be subnormal if it satisfies  
\begin{equation}\label{defsub}
  \sigma_f := \limsup_{r\rightarrow \infty} \frac{\log  T(r,f)}{r}=0,
\end{equation} 
 Their main result is stated as follows\cite[Theorem 1.1]{Nie-2025}: 
\begin{lettertheorem}\label{lemc}
  Let $f(z)$ be a subnormal transcendental meromorphic solution of the equation 
  \begin{equation}\label{eqi3}
    f(z+1)-f(z-1) + a(z)S(f,z)=R(z,f)=\frac{P(z,f)}{Q(z,f)},
  \end{equation}
  where $a(z)$ is rational, $P(z,f)$ and $Q(z,f)$ are coprime polynomials in $f$ with rational coefficients. 
  Then $\deg_fR\leq 7$, and $\deg_fP\leq \deg_fQ+1$. Moreover, if $Q(z,f)$ has a rational function root $b_1$ in $f$ with multiplicity $k$, then $k\leq 2$.  
  
\end{lettertheorem}

 Halburd and Korhonen\cite{Halb2-2007} studied the reduction of the equation 
\begin{equation}\label{eqi5}
  f(z+1)f(z-1)=\frac{c_2(f-c_{+})(f-c_{-})}{(f-a_{+})(f-a_{-})},
\end{equation}
where the coefficients are meromorphic functions.  They proved that if \eqref{eqi5} admits an admissible finite-order meromorphic solution with bounded pole multiplicity, then the equation reduces via M$\ddot{o} $bius transformation to canonical forms including difference Painlev$\acute{e} $ III, unless
$f$ is a solution of a difference Riccati equation.
Motivated by these results and Theorem \ref{lemc}, we replace the $f(z+1)-f(z-1)$ in \eqref{eqi3} with $f(z+1)f(z-1)$ and investigated reductions of  the resulting Schwarzian delay equation. Our main theorem follows:  

\begin{theorem}
  Let $f(z)$ be a subnormal transcendental meromorphic solution of the equation 
  \begin{equation}\label{eqm1}
    f(z+1)f(z-1)+a(z)S(f,z)=R(z,f)=\frac{P(z,f)}{Q(z,f)},
  \end{equation}
  where $a(z)$ is rational, $P(z,f)$ and $Q(z,f)$ are coprime polynomials in $f$ with rational coefficients. 
  Then $\deg_fR\leq 7$, and $\deg_fP\leq \deg_fQ+2$. Moreover, if $Q(z,f)$ has a rational function root $b_1$ in $f$ with multiplicity $k$, then $k\leq 2$.  
  
\end{theorem}

Below we give some examples to illustrate our results.  
\begin{exa}\label{exa1}
Let $f(z)=e^{\pi z}$. It is easy to check that 
\begin{equation*}
  f(z+1)f(z-1)+S(f,z)=f^2-\frac{\pi ^2}{2}.
\end{equation*}  
Here, $R(z,f)$ is a polynomial in $f$ with multiplicity no more than 2. 
\end{exa}

\begin{exa}\label{exa2}
Let $f(z)=e^{2\pi z}-z$. Then $f$ satisfies 
\begin{align*}
  f(z+1)f(z-1)+&S(f,z)= \frac{-8 \pi ^3[\pi f^2+(2\pi z+1)f+\pi z^2+z]}{(2\pi f+2\pi z-1)^2}+\\
   &[e^{2 \pi}f+(e^{2\pi}-1)z-1][e^{-2\pi}f+(e^{-2\pi}-1)z+1].
\end{align*}
Then $\deg_fP=\deg_fQ+2=4$, $Q(z,f)=(2\pi f+2\pi z-1)^2$, and $2\pi f+2 \pi z-1 $ has no multiple zeros. 
\end{exa}

\begin{exa}\label{exa3}
  It can be deduced that the meromorphic function $f(z)=1/(e^z-1)$ is a solution of the delay Schwarzian equation 
\begin{equation*}
  f(z+1)f(z-1)+S(f,z)=\frac{ef^2}{[(e-1)f+e][(1-e)f+1]}-\frac{1}{2}.
\end{equation*}
Here, $\deg_fP=\deg_fQ=2$. In addition, $f-e/(1-e)$ and $f-1/(e-1)$ have only simple zeros and hence have no multiple zeros. 
\end{exa}

\begin{exa}\label{exa4}
  Suppose that $f(z)= \wp (z+z_0 ; \omega _1 , \omega _2) $ is the Weierstrass elliptic function, where $\omega _1$ and $\omega _2$ are two fundamental periods that are linearly independent over $\mathbb{R} $. 
  Then $f(z)$ solves the equation 
  \begin{equation*}
    f'(z)^2=4f^3(z)-g_2f(z)-g_3,
  \end{equation*}
    where $g_2$ and $g_3$ are constants depending on $\omega _1$ and $\omega _2$. Then the Schwarzian derivative of $f$ is that 
    \begin{equation*}
      S(f,z)=\frac{-48f^4-24g_2f^2-96g_3f-3g_2^2}{32f^3-8g_2f-8g_3}. 
    \end{equation*}
  
    Then we choose the appropriate $g_2$ and $g_3$ such that $f'(1)=0$ and $4z^3-g_2z-g_3=0$ has only simple roots $e_1, e_2, e_3$, see \cite[example 1.1]{Ishi-2017}. 
    According to the addition theorem\cite[Chapter 20.3]{WhittakerBook} and the properties of Weierstrass elliptic functions, $f$ satisfies 
  \begin{align*}
    f(z+1)f(z-1)+S(f,z)&=\left[\frac{-(e_2e_3+e_1^2)-e_1f}{e_1-f}\right]^2\\
  &+\frac{-48f^4-24g_2f^2-96g_3f-3g_2^2}{8(f-e_1)(f-e_2)(f-e_3)}.
  \end{align*}
  Therefore, $\deg_fP=\deg_fQ+1$ and $e_1,e_2$ and $e_3$ are constant.  
  \end{exa}

Then we classify the case $\deg_fP\leq \deg_fQ+2$ when $Q(z,f)$ has only simple roots in $f$. 

  \begin{theorem}
  Let $f(z)$ be a subnormal transcendental meromorphic solution of \eqref{eqm1}. 
  Let $Q(z,f)$ be of the following form: 
  \begin{equation*}
    Q(z,f):=\prod_{i=1}^{n}(f(z)-b_i(z))\hat{Q}(z,f(z)), 
  \end{equation*}
  where the $b_i's(1\leq i\leq n, n\in\mathbb{N} )$ are distinct rational functions, $\hat{Q}(z,f)$ and $f-b_i$ are coprime. \\
$({\mathrm{I}})$ Assume that $ \deg_fP=\deg_fQ+1  $.  We have \\
 $(1)$ if $n=3$, then there exists an integer $i\in\{1,\cdots, n\}$ such that the root $b_i$ is a constant; \\
 $(2)$ if $n=2$, and $f-b_i$ has finitely many double zeros for all $i=1,\cdots, n$, then there exists a root $b_i$ which is non-constant.\\
 $({\mathrm{II}})$Suppose that $\deg_fP=\deg_fQ+2$. We also have \\
 $(3)$ if $n=2$,  then there exists an integer $i\in\{1,\cdots, n\}$ such that the root $b_i$ is a constant; \\
 $({\mathrm{III}})$ Suppose that $\deg_fP\leq \deg_fQ$. For at least two $i \in\{1,\cdots, n\}$, we have \\
$(4)$ if $n\geq 3$ and $b_i$ is non-constant,  then $f-b_i$ has finitely many multiple zeros;\\
$(5)$ if $n\geq 2$ and $b_i$ is constant, then $f-b_i$ has finitely many zeros with multiplicity at least 3. 
\end{theorem} 
\begin{remark}
  Example \ref{exa3} demonstrates that condition (5) in case~$(\mathrm{III})$ of Theorem 2 can be satisfied.  
  Example \ref{exa4} verifies that condition (2) in case~$(\mathrm{I})$ of Theorem 2 can not be omitted. 
\end{remark}

The following example shows that the case  $({\mathrm{II}})$ in Theorem 2 can occur and the roots  $b_i$ can be a constant for all $i\in\{1, \cdots, n\}$. 
\begin{exa}\label{exa5}
  Let $f(z)=\frac{1}{\sin(2\pi z)+\sqrt{2}i}$. Then $f(z)$ satisfies the following delay Schwarzian differential equation 
  \begin{equation*}
    f(z+1)f(z-1)+S(f,z)=f^2+2\pi^2\frac{-2\sqrt{2}if+1}{[(\sqrt{2}i+1)f-1][(\sqrt{2}i-1)f-1]}.
  \end{equation*}
  In this case, $\deg_fP=\deg_fQ+2$, and all roots of $Q(z,f)$ are constants. 
\end{exa}

Now we discuss a special situation when $Q(z,f)$ has a root of multiplicity 2 in the next theorem.

\begin{theorem}
  Let $f(z)$ be a subnormal transcendental meromorphic solution of \eqref{eqm1}. Let $Q(z,f)$ be of the following forms: 
  \begin{equation*}
    Q(z,f):=(f(z)-b_1(z))^2\hat{Q}(z,f(z))
  \end{equation*}
  where $b_1(z)$ is a rational function, $\hat{Q}(z,f)$ and $f-b_1$ are coprime.  If $\deg_fP=\deg_fQ+1 $ and $f-b_1$ has finite double zeros or if $\deg_fP = \deg_fQ+2,$ then $f-b_1$ has infinitely many simple zeros.  
 \end{theorem}
 \begin{remark}
  Example \ref{exa2} can shows the case of $\deg_fP=\deg_fQ+2$ in Theorem 3 can occur. 
 \end{remark}

\section{Proof of Theorem 1}

To prove Theorem 1, we first recall the necessary concepts and lemmas. A differential-difference polynomial in $f(z)$ is defined by 
\begin{equation*}
  P(z,f)=\sum_{l\in L} b_l(z)f(z)^{l_{0,0}}f(z+c_1)^{l_{1,0}}\cdots f(z+c_{\nu })^{l_{\nu ,0}}f'(z)^{l_{0,1}}\cdots f^{(\mu )}(z+c_{\nu })^{l_{\nu , \mu }},
\end{equation*}
where $c_1,\cdots, c_{\nu }$ are distinct complex non-zero constants, $L$ is a finite index set consisting of elements of the form $l=(l_{0,0}, \cdots,l_{\nu ,\mu })$ and coefficients $b_l$ are meromorphic functions for all $l\in L$. 
Nie, Huang , Wang and Wu modified the results of Halburd and Korhonen\cite[Lemma 2.1]{Halb-2017} and gave the following lemma\cite[Lemma 2.3]{Nie-2025}. 

\begin{lemma}\label{lema}
  Let $f(z)$ be a transcendental meromorphic solution of 
  \begin{equation*}
    P(z,f)=0,
  \end{equation*}
  where $P(z,f)$ is a differential-difference polynomial in $f$ with meromorphic coefficients $b_l(z)$ satisfying $m(r,b_l)=S(r,f)$. Let $a_1, \cdots , a_k$ be small functions to $f$ such that 
  $P(z,a_i)\not \equiv  0$ is also a small function of $f$, for all $i \in \{1, \cdots , k\}$. If there exists $s>0$ and $\tau \in (0, 1)$ such that 
  \begin{equation*}
    \sum_{i=1}^{k}n\left(r,\frac{1}{f-a_i}\right)\leq k \tau n(r+s, f)+O(1),
  \end{equation*}  
  then $\sigma_f>0$, where $\sigma_f$ is defined as \eqref{defsub}.
\end{lemma}

We present the following lemma, which estimates the degree of $R(z,f)$. 

  \begin{lemma}\label{lem1}
    Let $f(z)$ be a subnormal transcendental meromorphic solution of \eqref{eqm1}. Then $\deg_fR\leq 7$. Furthermore, if $R(z,f(z))$ is a polynomial in $f$ 
    with coefficients in $z$, then $\deg_f R\leq 2$.
  \end{lemma}

  \begin{proof}
    Taking the Nevanlinna characteristic function of both sides of \eqref{eqm1} and using \cite[Theorem~2.25]{LaineBook} we have
    \begin{align*}
      (\deg_fR)T(r,f) &= T(r,R(z,f))+O(\log r) \\
      &=T(r,f(z+1)f(z-1)+a(z)S(f,z))+ O(\log r) .
    \end{align*}
    From $T(r,f')\leq 2T(r,f)+S(r,f)$, \eqref{eqm1} and \cite[Theorem 1.2]{Zhen-2020}, we have
    \begin{align*}
      (\deg_fR)T(r,f) &\leq T(r,f(z+1)f(z-1))+T(r,S(f,z))+O(\log r)\\
      &\leq m\left(r,\frac{f(z+1)}{f(z)}\right)+m\left(r,\frac{f(z-1)}{f(z)}\right)+2m(r,f)\\
      &+N(r,f(z+1)f(z-1))+5T(r,f)+S(r,f)\\
      &\leq 7T(r,f)+S(r,f).
    \end{align*} 
    where $S(r,f)$ denotes any quantity satisfying $S(r,f)=o(T(r,f))$ as $r\rightarrow \infty$ and $r\not \in A_1$,  $A_1$ is a set with zero upper-density measure. 

    This gives the first conclusion of Lemma \ref{lem1}. Now we suppose $R$ is a polynomial in $f$ with rational coefficients. We will consider the following two cases.

    \textbf{Case 1} Suppose that $f(z)$ has at most finitely many poles. Since $R$ is a polynomial in $f$. From \eqref{eqm1}, we know that $N(r,S(f,z))=O (\log r)$. Then from the lemma on the logarithmic derivative we have  $T(r,S(f,z))=S(r,f)$ and   
    \begin{align*}
      (\deg_fR)T(r,f)&\leq T(r,f(z+1)f(z-1))+S(r,f)\\
      &\leq m\left(r,\frac{f(z+1)}{f(z)}\right)+m\left(r,\frac{f(z-1)}{f(z)}\right)+2m(r,f)+S(r,f)\\
      &\leq 2T(r,f)+S(r,f),
    \end{align*}
    where $r\not \in A_2$ and $A_2$ is a set with zero upper-density measure. So $\deg_fR\leq2$.

   \textbf{Case 2} Assume that $f(z)$ has infinitely many poles and we suppose that $\deg_fR\geq 3$.

   From conditions we know that the coefficients of $R(z,f)$ and $a(z)$ are rational, so they have finitely many zeros and poles. Therefore, there exists a constant $M \in \mathbb{R} $ such that all zeros and poles of $a(z)$ and the coefficients of $R(z,f)$ lie in a disk $|z|<M$ where $M\in \mathbb{R} ^{+}$.
   As $f(z)$ has infinitely many poles, there exists a point $z_0$ such that $|z_0|>M_1$, where $M_1$ is a positive constant such that $M_1>M$ and $z_0\pm n \in \{z: |z|>M_1\}$ for finite number $n\in \mathbb{Z} $. Let $z_0$ be a pole of $f(z)$ with multiplicity $t$. 
   Then $z_0$ is either a regular point or a double pole of $S(f,z)$. Since $R$ is a polynomial in $f$ and $\deg_fR\geq 3$, and we denote $\deg_fR=\alpha $ for convenience, it follows from \eqref{eqm1} that $f(z)$ may has poles at $z_0+1$ or $z_0-1$ and the number of the multiplicities of the poles is more than $t \alpha $.  We may assume that $f(z)$ takes a pole at $z_0+1$ with multiplicity $p_1$ and $z_0-1$ is a pole of $f$ with multiplicity $p_2$, where $p_1+p_2=t\alpha $. 
   
   \textbf{Subcase 2.1} Suppose that $p_1\alpha -t\leq 0$, then we can know that $p_1\leq \frac{t}{\alpha }$ and $p_2\geq \frac{t( \alpha ^2-1)}{ \alpha }\geq 2$. 
   From shifting \eqref{eqm1}, we have 
   \begin{equation*}
    f(z+2)f(z)+a(z+1)S(f,z+1)=R(z+1, f(z+1)),
   \end{equation*}
   from which it follows that $f(z_0+2)$ is finite. 
   
   If $0 \geq p_1 \alpha -t >-2$, $z_0+3$ is a zero of $f$ with multiplicity $p_1$. From the iteration of shifting \eqref{eqm1}, we find that there may be no poles of $f$ in the set $\{z_0+d\}$, where $d>3$ is an integer. 
   Now we calculate the multiplicities of poles of $f$ in the set $\{z_0-n\}$, where $n$ is a positive integer.  From the shifting of \eqref{eqm1}, similarly, we find $z_0-2$ is a pole of $f$ with multiplicity $t(\alpha^2 -1)-p_1 \alpha \geq 1 $.  We find that $z_0 -3$ is the pole of $f$ with multiplicity $t(\alpha ^3-2\alpha )+p_1(-\alpha ^2+1)\geq 1$ and  
  $z_0-4$ is a pole of $f$ with multiplicity $t(\alpha ^4-3\alpha ^2+1)+p_1(-\alpha ^3+2\alpha) \geq 1$. From the recurrence relation  we know that $z_0-d(d \in \mathbb{N} )$ is a pole of $f$, and the multiplicity of the pole of $f$ at $z_0-d$ is $\frac{2p_2 - t(\alpha  -\sqrt{\alpha ^2-4} )}{2\sqrt{\alpha ^2-4} }\cdot \left(\frac{\alpha +\sqrt{\alpha ^2-4} }{2}\right)^{d-1}+O\left(1\right)$. 
  Thus we have  
   \begin{equation*}
    n(d+|z_0|,f) \geq \left(\alpha-1\right)^{d-2}+O(1)
   \end{equation*}
   for sufficiently large $d \in \mathbb{N} $, then 
   \begin{align}\label{eqm5}
    \nonumber 0=\limsup_{r\rightarrow \infty} \frac{\log^+T(r,f)}{r}&\geq \limsup_{r\rightarrow \infty} \frac{\log n(r,f)}{r}\\
    \nonumber &\geq \limsup_{d\rightarrow \infty}\frac{\log n(d+|z_0|,f)}{d+|z_0|}\\
    &\geq \limsup_{d\rightarrow \infty}\frac{\log (\alpha -1)^{d-2}}{d+|z_0|}\\
    \nonumber &=\log (\alpha -1)\\
    \nonumber &\geq \log 2>0
   \end{align}
which is a contradiction. If $\alpha -t\leq -2$, $z_0+3$ is a simple pole of $f$ when $p_1=1$ or $f(z_0+3)$ is finite when $p_1\geq 2$. The number of multiplicities of poles of $f$ in the set $\{z_0-d, \cdots, z_0+d\}$ is more than the case when $0\geq \alpha -t>-2$. Similarly, we have the same conclusion. 

\textbf{Subcase 2.2} Assume that $p_1\alpha -t\geq 1$, then $p_1\geq \frac{t+1}{\alpha }$ and $z_0+2$ is a pole with multiplicity $p_1\alpha-t$. Similarly, from continuing the iteration of shifting of \eqref{eqm1}, we find that $z_0+3$ is a pole of $f$ with multiplicity $p_1\alpha ^2-t\alpha -p_1\geq 1$ and $z_0+4$ is a pole of $f$ with multiplicity $p_1\alpha ^3-t\alpha ^2-2p_1\alpha +t\geq 1$. 
If $2p_1-t(\alpha -\sqrt{\alpha ^2-4})>0$, from the recurrence relation again, we know that $z_0+d$ is a pole of $f$ and the multiplicity of the pole of $f$ at $z_0-d$ is $\frac{2p_1-t(\alpha -\sqrt{\alpha ^2-4}) }{2\sqrt{\alpha ^2-4} }\cdot \left(\frac{\alpha +\sqrt{\alpha ^2-4} }{2}\right)^{d-1}+O\left(1\right)$. If $2p_1-t(\alpha -\sqrt{\alpha ^2-4})\leq 0$, then we have $2p_2-t(\alpha -\sqrt{\alpha ^2-4})>0$. Similar to the subcase 2.1 we can have a contradiction again. 
Hence we complete the proof.
\end{proof}

Now we will consider the case when the polynomial $Q(z,f)$ has a rational function root in $f$.

\begin{lemma}\label{lem2}
  Let $f(z)$ be a subnormal transcendental meromorphic solution of \eqref{eqm1}. Let $R(z,f)$ be of the following form:
  \begin{equation}\label{eqm2}
    R(z,f(z)) := \frac{P(z,f(z))}{\left(f(z)-b_1(z)\right)^k\hat{Q}(z,f(z))}
  \end{equation}
  where $k$ is a positive integer, $b_1$ is a rational function, $P(z,f)$ and $\hat{Q}(z,f)$ are polynomials in $f$ with rational coefficients, then $k\leq 2$.
\end{lemma}
  \begin{proof}
Suppose a differential-difference polynomial
\begin{equation*}
  \Psi (z,w) := (w-b_1)^k\hat{Q}(z,w)(w(z+1)w(z-1)+\hat{a}(z))-P(z,w),
\end{equation*}    
where $k\geq1$, and $\hat{a}(z):=a(z)S(w,z)$ satisfies $m(r,\hat{a})=S(r,w)$. 
Clearly, from \eqref{eqm1} and \eqref{eqm2}, we have  $\Psi (z,f) =0$. 
Since $\Psi (z,b_1)\neq 0$ is rational, the first condition of Lemma \ref{lema} is satisfied 
  for $b_1$. By similar to the proof of \cite[Lemma 2.3]{Nie-2025}, we have 
  \begin{equation*}
    N\left(r,\frac{1}{f-b_1}\right)=T(r,f)+S(r,f).
  \end{equation*} 
      
Now we consider the multiplicities of zeros of $f-b_1$ and poles of $f$. We assume that $z_0$ is a zero (or pole) of $f-b_1$ with multiplicity $p(\geq 1)$ and that none of $a(z), b_1(z)$ and any 
rational coefficients in $R(z,f)$ have a zero or pole at $z_0$. We can also require that for finite number $n\in \mathbb{Z}$, the shifting points $z_0+n$ are not the zeros 
or poles of those rational coefficient functions including $a(z)$ and $b_1(z)$.  
We call such a point $z_0$ a generic zero (or generic pole) of $f-b_1$ with multiplicity $p$, see in \cite[Lemma 3.2]{Nie-2025}.  

Assume that $k\geq 3$.  Since $k\geq 3$, then $kp\geq 3p>2$. Thus, $f(z+1)f(z-1)$ has a pole 
with multiplicity $kp(\geq 3)$ at $z_0$ from \eqref{eqm1} and \eqref{eqm2}. 

\textbf{Case 1} Assume that
\begin{equation*}
  \deg_f(P)\leq k+\deg_f(\hat{Q}).
\end{equation*}

\textbf{Subcase 1.1} 
Suppose the set $E_1=\{z:f(z+1)=\frac{g_1(z)}{(z-z_0)^q},\ f(z-1)\neq \infty\}$, where $g_1(z_0)\neq 0, \infty$ in $E_1$ and $q(\geq kp)$ is an integer. 
If $z_0\in E_1$, then $z_0+1$ is a double pole of $S(f,z)$.
By shifting \eqref{eqm1} we obtain 
\begin{align*}
  \nonumber &f(z+2)f(z)+a(z+1)S(f,z+1)\\
&=\frac{P(z+1,f(z+1))}{(f(z+1)-b_1(z+1))^k\hat{Q}(z+1, f(z+1))}.
\end{align*}
Thus $f(z)$ has a double pole at $z_0+2$, and so $z_0+2$ is a double pole of $S(f,z)$. By iterating \eqref{eqm1} again, we have 
\begin{align*}
  \nonumber &f(z+3)f(z+1)+a(z+2)S(f,z+2)\\
  &=\frac{P(z+2,f(z+2))}{(f(z+2)-b_1(z+2))^k\hat{Q}(z+2,f(z+2))}.
\end{align*}  
Then $f(z)$ has a zero at $z_0+3$ with multiplicity $q-2$. 

If $q=3$, then $z_0+4$ is a double zero of $f$, so $z_0+4$ is a double pole of $S(f,z)$. Then $z_0+5$ is a pole of $f$ with multiplicity 3, $z_0+6$ is a pole of $f$ with multiplicity 4, $z_0+7$ is a simple zero of $f$ and $z_0+8$ is a zero of $f$ with multiplicity 4. Thus we can know that $z_0+d$ is a zero or pole of $f$ for any $d\in \mathbb{N} $.  

If $q\geq 4$, then $z_0+4$ may be another generic zero of $f-b_1$ such that $z_0+5$ is a pole of $f$. Suppose that the multiplicity of the zero of $f-b_1$ at $q_1(\geq 1)$, thus we can know that $z_0+5$ is a pole of $f$ with multiplicity $kq_1+q-2(\geq 3)$, $z_0+6$ is a double pole of $f$ and $z_0+7$ is a zero of $f$. Suppose that the multiplicity of the zero of $f$ at $z_0+7$ is $q_2(\geq 1)$. When $q_2\geq 2$, $z_0+8$ may be the generic zero of $f-b_1$. 
Similar to the above analysis, $z_0+9$ and $z_0+10$ are poles of $f$, $z_0+11$ is a zero of $f$. 
When $q_2=1$, similar to the case when $q\geq 4$, $z_0+d$ is a pole or zero of $f$ for any $d\in\mathbb{N} $ and $d\geq 5$. 

From above the discussion, we have 
\begin{equation*}
  n_{E_1}\left(r,\frac{1}{f-b_1}\right)\leq \frac{1}{k}n_{E_1}(r+1,f)+S(r,f),
\end{equation*}
where $n_{E_i}(r,f)$ is the number of multiplicities of all poles of $f$ in the set $E_i\bigcap \{z\ :\ |z|<r\}$ for $i\in \mathbb{N} $. The definition of $E_i(i\geq 2)$ will give in the following subcases.

\textbf{Subcase 1.2} Assume that $E_2=\{z:f(z+1) =\frac{g_2(z)}{z-z_0}, f(z-1)= \frac{g_3(z)}{(z-z_0)^{q-1}}\}$, where $g_m(z_0)\neq 0, \infty$ in $E_2$ for $m\in \{2,3\}$  and $q-1=kp-1\geq 2$. If $z_0 \in E_2$, then $z_0+2$ may be the generic zero of $f-b_1$, $z_0-2$ is a double pole of $f$ and $f(z_0-3)$ is finite. If $z_0-3$ is a zero of $f$, we suppose that the multiplicity of the zero of $f$ at $z_0-3$ is $q_3(\geq 1)$. When $q_3\geq 2$, $z_0-4$ may be the generic zero of $f-b_1$. When $q_3=1$, 
similar to the subcase 1.1, $z_0-d$ is a zero or pole of $f$ for any $d\in \mathbb{N} $. 

From the discussion in this subcase, we obtain  
\begin{equation*}
  n_{E_2}\left(r,\frac{1}{f-b_1}\right)\leq \frac{1}{k}n_{E_2}(r+1,f)+S(r,f).
\end{equation*}

\textbf{Subcase 1.3} Suppose that $E_3=\{z:f(z+1)=\frac{g_4(z)}{(z-z_0)^2}, f(z-1)=\frac{g_5(z)}{(z-z_0)^{q-2}}\}$, where $g_m(z_0)\neq 0, \infty$ in $E_3$ for $m\in \{4,5\}$  and  $q-2=kp-2\geq 1$. If $z_0 \in E_3$, by shifting of \eqref{eqm1}, we find that $z_0+2$ is a double pole of $f$ and $z_0+3$ may be the generic zero of $f-b_1$. 
When $q-2=1$, it's similar to the subcase 1.2, we have $z_0-2$ may be the generic zero of $f-b_1$. When $q-2=2$, we have $z_0-2$ is a double pole of $f$ and $f(z_0-3)$ is finite. When $q-2\geq 3$, it's similar to the subcase 1.1. 
From the discussion in subcase 1.3, we have 
\begin{equation*}
  n_{E_3}\left(r,\frac{1}{f-b_1}\right)\leq \frac{1}{k}n_{E_3}(r+1,f)+S(r,f).
\end{equation*}

\textbf{Subcase 1.4} Assume that $E_4=\{z:f(z+1)=\frac{g_6(z)}{(z-z_0)^{s_1}},f(z-1)=\frac{g_7(z)}{(z-z_0)^{s_2}}\}$, where $g_m(z_0)\neq 0, \infty$ in $E_4$ for $m\in \{6,7\}$ , $s_1\geq 3$, $s_2 \geq 3$ and $s_1+s_2=kp$. It's similar to the subcases 1.1, 1.2 and 1.3, we have 
\begin{equation*}
  n_{E_4}\left(r,\frac{1}{f-b_1}\right)\leq \frac{1}{k}n_{E_4}(r+1,f)+S(r,f).
\end{equation*}

By adding up the contribution from all points in $\bigcup_{i=1}^4 E_i$ to the corresponding counting functions, it follows that 
\begin{equation*}
  n\left(r,\frac{1}{f-b_1}\right)=\sum_{i=1}^{4}n_{E_i}\left(r,\frac{1}{f-b_1}\right)\leq \frac{1}{k}n(r+1,f)+S(r,f).
\end{equation*}

 Thus both conditions of Lemma \ref{lema} are satisfied, and so $\sigma_f >0$. This is a contradiction. 

\textbf{Case 2} Assume that
 \begin{equation*}
  \deg_fP\geq k+\deg_f(\hat{Q})+1
 \end{equation*}
and $\deg_fR=\alpha(\geq 1)$. 

\textbf{Subcase 2.1} Suppose the set $E_5=\{z:f(z+1)=\frac{g_8(z)}{(z-z_0)^q},\ f(z-1)\neq \infty\}$ and $z_0\in E_5$, where $g_8(z_0)\neq 0, \infty$ in $E_5$ and $q\geq kp \geq 3$. 

If $\alpha \geq 2$, we continue to iterate the shifting and find that $z_0+2$ is a pole of $f$ with multiplicity $q\alpha $, $z_0+3$ is a pole of $f$ with multiplicity $q(\alpha ^2-1 )\geq 2$, $z_0+4$ is a pole of $f$ with multiplicity $q(\alpha ^3-2\alpha)\geq 2$ and $z_0+5$ is $q(\alpha ^4-3\alpha ^2+1 )\geq 2$. 
From the recurrence relation we can know that $z_0+d$ is a multiple pole of $f$ for any $d\in\mathbb{N} $. 

If $\alpha =1$, then $z_0+2$ is a pole of $f$ with multiplicity $q$, $z_0+3$ may be another zero of $f-b_1$. 

From the above analysis, we have 
\begin{equation*}
  n_{E_5}\left(r,\frac{1}{f-b_1}\right)\leq \frac{1}{k}n_{E_5}(r+1,f)+S(r,f),
\end{equation*}
where $n_{E_5}(r,f)$ is the number of multiplicities of all poles of $f$ in the set $E_5\bigcap \{z\ :\ |z|<r\}$. 

\textbf{Subcase 2.2} Assume that $E_6=\{z:f(z+1)=\frac{g_9(z)}{(z-z_0)^{q-1}}, f(z-1)=\frac{g_{10}(z)}{z-z_0}\}$ and $z_0\in E_6$, where $g_m(z_0)\neq 0, \infty$ in $E_6$ for $m\in \{9,10\}$  and $q-1=kp-1\geq 2 $. 

If $\alpha =1$, from continuing iteration of shifting, we have that $z_0-2$ is a simple pole and $z_0-3$ may be another generic zero of $f-b_1$. 
When $q-1=2$, $z_0+2$ may be the generic zero of $f-b_1$. When $q-1\geq 3$, similar to the subcase 2.1, $z_0+2$ is a pole of $f$ with multiplicity $q-1$ and $z_0-3$ may be another zero of $f-b_1$.  

If $\alpha >1$, we find that $z_0-2$ is a pole of $f$ with multiplicity $\alpha \geq 2$, $z_0-3$ is a pole of $f$ with multiplicity $\alpha ^2-1\geq 2$ and $z_0-4$ is a pole of $f$ with multiplicity $\alpha ^3-3\alpha \geq2$. From the recurrence relation again, we find that $z_0\pm d$ is a multiple pole of $f$ for any $d\in\mathbb{N} $.

Thus, we obtain 
\begin{equation*}
  n_{E_6}\left(r,\frac{1}{f-b_1}\right)\leq \frac{1}{k}n_{E_6}(r+1,f)+S(r,f).
\end{equation*}

\textbf{Subcase 2.3} Suppose that $E_7=\{z:f(z+1)=\frac{g_{11}(z)}{(z-z_0)^{q-2}}, f(z-1)=\frac{g_{12}(z)}{(z-z_0)^2}\}$ and $z_0\in E_7$, where $g_m(z_0)\neq 0, \infty$  in $E_7$ for $m\in \{11,12\}$ and $q-2=kp-2\geq 1$. 

If $\alpha =1$, then $z_0-2$ may be a generic zero of $f-b_1$. When $q-2=1$ or $q-2\geq 3$ , similar to the subcase 2.2 we have that $z_0+2$ is a pole of $f$ with multiplicity $q-2$ and $z_0+3$ is a generic zero of $f-b_1$. 
When $q-2=2$, $z_0+2$ may be a generic zero of $f-b_1$.  

If $\alpha >1$, it similar to the subcase 2.1 and we can obtain that $z_0\pm d$ is a pole of $f$ for any $d\in\mathbb{N} $. 

From the above,  when $\alpha =1$ we have 
\begin{equation*}
  n_{E_7}\left(r,\frac{1}{f-b_1}\right)\leq \frac{2}{k}n_{E_7}(r+1,f)+S(r,f).
\end{equation*}
When $\alpha >1$, we obtain 
\begin{equation*}
  n_{E_7}\left(r,\frac{1}{f-b_1}\right)\leq \frac{1}{k}n_{E_7}(r+1,f)+S(r,f).
\end{equation*}

\textbf{Subcase 2.4} Suppose that $E_8=\{z:f(z+1)=\frac{g_{13}(z)}{(z-z_0)^{s_1}}, f(z-1)=\frac{g_{14}(z)}{(z-z_0)^{s_2}}\}$ and $z_0\in E_8$, where $g_m(z_0)\neq 0, \infty$ in $E_8$ for $m\in \{13,14\}$  and $s_1+s_2=kp$. Similar to subcase 2.1, we have 
\begin{equation*}
  n_{E_8}\left(r,\frac{1}{f-b_1}\right)\leq \frac{1}{k}n_{E_8}(r+1,f)+S(r,f).
\end{equation*}

By adding up the contribution from all points in $\bigcup_{i=1}^4 E_i$ to the corresponding counting functions, 
if $\alpha >1$, we get 
\begin{equation*}
  n\left(r,\frac{1}{f-b_1}\right)=\sum_{i=5}^{8}n_{E_i}\left(r,\frac{1}{f-b_1}\right)\leq \frac{1}{k}n(r+1,f)+O(1). 
\end{equation*} 
If $\alpha =1$, we have 
\begin{equation*}
  n\left(r,\frac{1}{f-b_1}\right)=\sum_{i=5}^{8}n_{E_i}\left(r,\frac{1}{f-b_1}\right)\leq \frac{2}{k}n(r+1,f)+O(1).
\end{equation*}
Thus, by Lemma \ref{lema}, we have $\sigma _f>0$, a contradiction.
\end{proof}

\begin{proof2}
 We know that $\deg_fR=\max \{\deg_fP, \deg_fQ\}\leq 7$ from Lemma \ref{lem1}, This completes the proof of the first part. 

Assume that $\deg_fP\geq \deg_fQ+3$, then $\deg_fR=\deg_fP \geq 3$ and suppose that $\alpha =\deg_fR$ for convenience. We divide two cases as follows. 

\textbf{Case 1} Suppose that $f(z)$ has finitely many poles. From \eqref{eqm1}, we see that $N(r,S(f,z))=O(\log r)$. Then we have $T(r,S(f,z))=S(r,f)$ and 
\begin{align*}
  (\deg_fR)T(r,f)&\leq T(r,f(z+1)f(z-1))+T(r,S(f,z))+S(r,f)\\
  &\leq m\left(r,\frac{f(z+1)}{f(z)}\right)+m\left(r,\frac{f(z-1)}{f(z)}\right)+2m(r,f)+S(r,f)\\
  &\leq 2T(r,f)+S(r,f),
\end{align*}
where $r\not \in A_3$, where $A_3$ is a set with zero upper-density measure.
 So it contradicts to $\deg_fR\geq 3$.

\textbf{Case 2} Assume that $f(z)$ has infinitely many poles. 
Let $z_0$ be a generic pole of $f(z)$ with multiplicity $p$. 
Then $R(z,f)$ has a pole at $z_0$ with multiplicity $p\alpha (\geq 3)$, and $z_0$ is either a regular point or a double pole of $S(f,z)$. 
So $f(z+1)f(z-1)$ has a pole at $z_0$ with multiplicity $p\alpha$.

Suppose that $z_0+1$ is a pole of $f$ with multiplicity  $p_1(\geq 0) $, $z_0-1$ is pole of $f$ with multiplicity $p_2(\geq 0)$, where $p_1+p_2=p\alpha $. 

\textbf{Subcase 2.1} Assume that $p_1\alpha -p\leq 0$, then we have that $p_2\geq \frac{p(\alpha ^2-1)}{\alpha }$ and $z_0-d$ is a pole of $f$ from the recurrence relation. 
For sufficiently large $d\in \mathbb{N} $ we find that  the multiplicity of the pole of $f$ at $z_0-d$ is 
\begin{equation*}
\frac{2p_2 - p(\alpha  -\sqrt{\alpha ^2-4} )}{2\sqrt{\alpha ^2-4} }\cdot \left(\frac{\alpha +\sqrt{\alpha ^2-4} }{2}\right)^{d-1}+O\left(1\right).
\end{equation*}
Thus, similar to the proof of Lemma  \ref{lem1}, we have 
\begin{equation*}
  n(d+|z_0|,f)\geq (\alpha -1 )^{d-2}+O(1)
\end{equation*} 
for all sufficiently large $d\in \mathbb{N}$. 
In this subcase, we also have equation \eqref{eqm5} holds, which is a contradiction.

\textbf{Subcase 2.2} Suppose that $p_1\alpha-p\geq 1$, similar to the proof in Lemma \ref{lem1}, we have a contradiction again.

\end{proof2}

\section{Proof of Theorem 2}

Obviously, none of $b_1, b_2, \cdots, b_n$ is a solution of \eqref{eqm1} and $\deg_fQ=n+\deg_f\hat{Q}$. We consider 
three cases as follows. 

\textbf{Case 1} Suppose that $\deg_fP=\deg_fQ+1$. 

\textbf{(1)} Assume that the $b_i's$ are non-constant for all $i\in\{1,\cdots,n\}$. We will derive a contradiction. Without loss of generality, let $z_0$ be a generic zero of $f-b_1$ with multiplicity $q$ such that $b_1'(z_0)\neq 0$. 
Firstly, we will prove $f'(z_0)=0$ is impossible. If $f'(z_0)=0$, then we find that $q=1$, $z_0$ is a double pole of $S(f,z)$ and $f(z+1)f(z-1)$ has a double pole at $z_0$. 

\textbf{Subcase 1.1 } Suppose that $U_1=\{z:f(z+1)=\frac{h_1(z)}{(z-z_0)^{q_1}},f(z-1)\neq \infty\}$ and $z_0\in U_1$, where $h_1(z)\neq 0, \infty$ in $U_1$ and $q_1\geq 2$. 

When $q_1=2$, then $z_0+2$ may be another zero of $f-b_i$. If $f(z_0+2)\neq b_i(z_0+2)$ for $i \in \{1,\cdots, n\}$, then $f(z+3)$ has a double zero at $z_0$, $z_0+4$ is a double pole of $f$ and $z_0+5$ may be the generic zero of $f-b_i$.  If $f(z_0+2)=b_i(z_0+2)$ for $i \in \{1,\cdots , n\}$, we assume that the multiplicity of the  generic zero of $f-b_i$ at $z_0+2$ is $p(\geq 1)$. When $p=1$, for $f'(z_0+2)\neq 0$, $z_0+3$ is a simple zero of $f$ and $z_0-4$ may be the generic zero of $f-b_i$. For $f'(z_0+2)=0$,  $f(z_0+3)$ is finite. 
When $p=2$, we also have $f(z_0+3)$ is finite.  When $p=3$, $z_0+3$ and $z_0+4$ are simple poles of $f$, $z_0+5$ may be the generic zero of $f-b_i$. 
When $p=4$, $z_0+3$ is a double pole of $f$ and $z_0+4$ may be the generic zero of $f-b_i$. When $p\geq 5$, $f(z_0+3)$ and $f(z_0+4)$ are poles of $f$. When $q_1\geq 3$, it's similar to the situation when $q_1=2$. 

From the discussion above, we have 
\begin{equation*}
  \sum_{i=1}^{n} n_{U_1}\left(r,\frac{1}{f-b_i}\right)\leq 2n_{U_1}(r+1,f)+O(1), 
\end{equation*}
where $n_{U_i}(r,f)$ is the number of multiplicities of all poles of $f$ in the set $U_i\bigcap \{z\ :\ |z|<r\}$ for $i\in \mathbb{N} $. The definition of $U_i(i\geq 2)$ will give in the following subcases.

\textbf{Subcase 1.2} Assume that $U_2=\{z:f(z+1)=\frac{h_2(z)}{z-z_0}, f(z-1)=\frac{h_3(z)}{z-z_0}\}$ and $z_0\in U_2$, where $h_m(z)\neq 0, \infty$ in $U_2$ for $m\in \{2,3\}$,  then we have $z_0\pm 2$ are simple poles of $f$ by shifting \eqref{eqm1} and $z_0\pm 3$ are another generic zeros of $f-b_1$.  
In this situation, we can also get equation 
\begin{equation*}
  \sum_{i=1}^{n} n_{U_2} \left(r,\frac{1}{f-b_i}\right)\leq  n_{U_2}(r+1,f)+O(1). 
\end{equation*}
 
From the above analysis, we have 
\begin{align*}
  \sum_{i=1}^{n} n \left(r,\frac{1}{f-b_i}\right)&=\sum_{i=1}^{n}\left(n_{U_1}\left(r,\frac{1}{f-b_i}\right)+n_{U_2}\left(r,\frac{1}{f-b_i}\right)\right)\\
  &\leq  2n(r+1,f)+O(1). 
\end{align*}
Thus $\sigma _f>0$ and we get a contradiction by Lemma \ref{lem1}.   

Therefore, it's clear that $f'(z_0)\neq 0$. In this case, $z_0$ is a regular point of $S(f,z)$ and a pole of $f(z+1)f(z-1)$ with multiplicity $q(\geq 2)$. 

\textbf{Subcase 1.3} Assume that $U_3=\{z:f(z+1)=\frac{h_4(z)}{(z-z_0)^2}, f(z-1)\neq \infty\}$ and $z_0\in U_3$, where $h_4(z)\neq 0, \infty$ in $U_3$, then we have that $z_0+2$ may be the generic zero of $f-b_i$. It follows that 
\begin{equation*}
  \sum_{i=1}^{n} n_{U_3} \left(r,\frac{1}{f-b_i}\right)\leq 2 n_{U_3}(r+1,f)+O(1). 
\end{equation*}

\textbf{Subcase 1.4} Suppose that $U_4=\{z:f(z+1)=\frac{h_5(z)}{z-z_0}, f(z-1)=\frac{h_6(z)}{z-z_0}\}$ and $z_0\in U_4$, where $h_m(z)\neq 0, \infty$ in $U_4$ for $m\in \{5,6\}$, then $z_0 \pm 2$ are simple poles of $f$ and $z_0\pm 3$ may be the generic zeros of $f-b_i$. Thus we obtain 
\begin{equation*}
  \sum_{i=1}^{n} n_{U_4} \left(r,\frac{1}{f-b_i}\right)\leq  n_{U_4}(r+1,f)+O(1). 
\end{equation*}

\textbf{Subcase 1.5} Assume that $U_5=\{z:f(z_0+1)=\frac{h_7(z)}{(z-z_0)^{q_2}}, f(z_0-1)\neq \infty\}$ and $z_0\in U_5$, where $h_7(z)\neq 0, \infty$ in $U_5$ and $q_2\geq 3$. By iteration of shifting, we can know that $z_0+2$ is a pole of $f$ with multiplicity $q_2$ and $z_0+3$ may be another zero of $f-b_i$ such that $f(z_0+4)$ is finite. 
Therefore, it follows that 
\begin{equation*}
  \sum_{i=1}^{n} n_{U_5} \left(r,\frac{1}{f-b_i}\right)\leq  n_{U_5}(r+1,f)+O(1). 
\end{equation*}

\textbf{Subcase 1.6} Suppose that $U_6=\{z: f(z+1)=\frac{h_8(z)}{(z-z_0)^{q-1}}, f(z-1)=\frac{h_9(z)}{z-z_0}\}$ and $z_0\in U_6$, where $h_m(z)\neq 0, \infty$ in $U_6$ for $m\in \{8,9\}$  and $q\geq 3$. Similar to the subcase 1.4, $z_0-2$ is a simple pole of $f$ and $z_0-3$ may be the generic zero of $f-b_i$. When $q-1=2$, it's similar to the subcase 1.3 and we have $z_0+2$ may be the generic zero of $f-b_i$. 
When $q-1\geq 3$, we find that $z_0+2$ is a pole of $f$ and $f(z_0+3)$ is finite. From the above analysis, we have 
\begin{equation*}
  \sum_{i=1}^{n} n_{U_6} \left(r,\frac{1}{f-b_i}\right)\leq  2n_{U_6}(r+1,f)+O(1). 
\end{equation*}

 \textbf{Subcase 1.7} Assume that $U_7=\{z:f(z+1)=\frac{h_{10}(z)}{(z-z_0)^{q-2}}, f(z-1)=\frac{h_{11}(z)}{(z-z_0)^2}\}$ and $z_0\in U_7$, where $h_m(z)\neq 0, \infty$ in $U_7$ for $m\in \{10, 11\}$ and $q\geq 3$, then $z_0-2$ may be a generic zero of $f-b_i$ with multiplicity 2, $i\in \{1,\cdots, n\}$. Similar to the subcases 1.3-1.6, we have 
\begin{equation*}
  \sum_{i=1}^{n} n_{U_7} \left(r,\frac{1}{f-b_i}\right)\leq  2n_{U_7}(r+1,f)+O(1). 
\end{equation*}

\textbf{Subcase 1.8} Suppose that $U_8=\{z:f(z+1)=\frac{h_{12}(z)}{(z-z_0)^{s_1}}, f(z-1)=\frac{h_{13}(z)}{(z-z_0)^{s_2}}\}$ and $z_0\in U_8$, where $h_m(z)\neq 0, \infty$ in $U_8$ for $m\in \{12,13\}$, $s_1 \geq 3$, $s_2\geq 3$ and $s_1+s_2=q$. Similar to the subcase 1.5, we find that 
\begin{equation*}
  \sum_{i=1}^{n} n_{U_8} \left(r,\frac{1}{f-b_i}\right)\leq  n_{U_8}(r+1,f)+O(1). 
\end{equation*}

From the above discussion, we can conclude that 
\begin{equation*}
  \sum_{i=1}^{n} n \left(r,\frac{1}{f-b_i}\right)=\sum_{i=1}^{n}\sum_{j=3}^{8}n_{U_j}\left(r,\frac{1}{f-b_i}\right)\leq  2n(r+1,f)+O(1). 
\end{equation*}
Then we have a contradiction from the Lemma \ref{lem1}. 

\textbf{(2)} Suppose that the $b_i'$s are constant for all $i\in\{1,2,\cdots,n\}$. Then by the condition and Lemma \ref{lema}, 
$f-b_i$ has infinitely many zeros with multiplicity 1 or not less than 3.
Without loss of generality, we may assume that $z_1$ is a generic zero of $f-b_1$ with multiplicity s, where $s\geq 3$ or $s=1$. If $s=1$, then $f'(z)\neq 0$ and so $S(f,z)$ is regular at $z_1$. If $s\geq 3$, then 
$f'(z_1)=0$ and $S(f,z)$ has a double pole at $z_1$. Hence, for each case, we conclude that $f(z+1)f(z-1)$ has a pole with multiplicity $s$ at $z_1$.  
Similar to the proof of subcases 1.1-1.8, we also get
\begin{equation}\label{eqm4}
  \sum_{i=1}^{n} n \left(r,\frac{1}{f-b_i}\right)\leq  n(r+1,f)+O(1). 
\end{equation}
Therefore, it follows that $\sigma _f>0$, a contradiction from Lemma \ref{lem1}. 

\textbf{Case 2} Suppose that $\deg_fP=\deg_fQ+2$. 

\textbf{(3)}Assume that the $b_i$'s are non-constant for all $i \in \{1,\cdots, n\}$. We can let $z_0$ be a generic zero of $f-b_1$ with multiplicity $q$ such that $b_1'(z_0)\neq 0$ without loss of generality. If $f'(z_0)=0$, then $q=1$ and $z_0$ is a double pole of $S(f,z)$, so 
$f(z+1)f(z-1)$ has a double pole at $z_0$. 

\textbf{Subcase 2.1} Suppose that $U_9= \{z:f(z+1)=\frac{h_{14}(z)}{(z-z_0)^{q_3}}, f(z-1)\neq \infty\}$, where $h_{14}(z)\neq 0, \infty$ in $U_9$ and $q_3\geq 2$. If $z_0\in U_9$, then $z_0+2$ is a pole of $f$ with multiplicity at least 4, $z_0+3$ is a pole of $f$ with multiplicity at least 6 and $z_0+4$ is a pole of $f$ with multiplicity at least 8. From the recurrence relation, we have that $z_0+d$ is a multiple pole of $f$ with multiplicity at least $2d$ for any $d\in \mathbb{N} $. 
Therefore, it follows that 
\begin{equation*}
  \sum_{i=1}^{n} n_{U_9} \left(r,\frac{1}{f-b_i}\right)\leq\frac{1}{2} n_{U_9}(r+1,f)+O(1). 
\end{equation*}

\textbf{Subcase 2.2} Assume that $U_{10}=\{z:f(z+1)=\frac{h_{15}(z)}{z-z_0}, f(z-1)=\frac{h_{16}(z)}{z-z_0}\}$ and $z_0\in U_{10}$, where $h_m(z)\neq 0, \infty$ in $U_{10}$ for $m\in \{15,16\}$. Then $z_0\pm 2$ are double poles of $f$, $z_0\pm 3$ are poles of $f$ with multiplicity 3 and $z_0\pm 4$ are poles of $f$ with multiplicity 4. Similar to the subcase 2.1, $z_0\pm d$ are poles of $f$ for any $d\in \mathbb{N} $. Thus we have 
\begin{equation*}
  \sum_{i=1}^{n} n_{U_{10}} \left(r,\frac{1}{f-b_i}\right)\leq \frac{1}{2} n_{U_{10}}(r+1,f)+O(1).  
\end{equation*}

Therefore, we add up the contribution from all points $z$ in the set $U_9\bigcup U_{10}$ to the corresponding counting functions, we have 
\begin{align*}
  \sum_{i=1}^{n} n \left(r,\frac{1}{f-b_i}\right)&=\sum_{i=1}^{n}\left(n_{U_9}\left(r,\frac{1}{f-b_i}\right)+n_{U_{10}}\left(r,\frac{1}{f-b_i}\right)\right)\\
 & \leq \frac{1}{2} n(r+1,f)+O(1). 
\end{align*}
 
 So it follows that $\sigma _f>0$ by Lemma \ref{lema}, a contradiction. From the discussion above, we have $f'(z_0)\neq 0$. In this case, $z_0$ is a regular point of $S(f,z)$ and a pole of $f(z+1)f(z-1)$ with multiplicity $q$. Similar to the subcases 2.1 and 2.2, we can have the same inequality \eqref{eqm4}. 
 thus we get a contradiction by Lemma \ref{lema}.

\textbf{Case 3} Suppose that $\deg_fP\leq \deg_fQ.$

\textbf{(4)} Assume that $f-b_i$ has infinitely many multiple zeros for $i\in \{1,\cdots, n\}$. Then we will prove by contradiction. Without loss of generality, let $z_0$ be a generic zero of $f-b_1$ with multiplicity $p(\geq 2)$ such that 
$b_1'(z_0)\neq 0$. Then $f'(z_0)\neq 0$ and $f(z+1)f(z-1)$ has a pole with multiplicity $p$ at $z_0$. 

\textbf{Subcase 3.1} Suppose that $U_{11}=\{f(z_0+1)=\frac{h_{17}(z)}{(z-z_0)^{q_4}}, f(z_0-1)\neq \infty\}$ and $z_0\in U_{11}$, where $h_{17}(z)\neq 0, \infty$ in $U_{11}$ and $q_4\geq p$, then $f(z+2)$ has a double pole at $z_0$.  
When $q_4=2$, we have $z_0+3$ may be the generic zero of $f-b_i$. 
When $q_4=3$, then $z_0+3$ is a simple zero of $f$, $z_0+4$ is a double zero of $f$, $z_0+5$ is a pole of $f$ with multiplicity 3 and $z_0+6$ is a pole of $f$ with multiplicity 4. From continuing the iteration of shifting, we know that $z_0+d$ is not a generic zero of $f-b_i$ for all $d\in \mathbb{N} $.    
When $q_4\geq 4$, we find that $z_0+3$ is a zero of $f$ with multiplicity $q_4-2(\geq 2)$ and $z_0+4$ may be another generic zero of $f-b_1$. 
From the above analysis, we obtain 
\begin{equation*}
  \sum_{i=1}^{n} n_{U_{11}} \left(r,\frac{1}{f-b_i}\right)\leq  n_{U_{11}}(r+1,f)+O(1), 
\end{equation*}

\textbf{Subcase 3.2} Assume that $U_{12}=\{z:f(z+1)=\frac{h_{18}(z)}{z-z_0}, f(z-1)=\frac{h_{19}(z)}{z-z_0}\}$ and $z_0\in U_{12}$, where  $h_m(z)\neq 0, \infty$ in $U_{12}$ for $m\in \{18,19\}$ , by iteration of shifting \eqref{eqm1}, we find that $z_0\pm 2$ may be generic zeros of $f-b_i$. 
Thus we have 
\begin{equation*}
  \sum_{i=1}^{n} n_{U_{12}} \left(r,\frac{1}{f-b_i}\right)\leq  2n_{U_{12}}(r+1,f)+O(1).  
\end{equation*}

\textbf{Subcase 3.3} Assume that $U_{13}=\{f(z+1)=\frac{h_{20}(z)}{(z-z_0)^{s_1}}, f(z-1)=\frac{h_{21}(z)}{(z-z_0)^{s_2}}\}$ and $z_0\in U_{13}$, where  $h_m(z)\neq 0, \infty$ in $U_{13}$ for $m\in \{20, 21\}$ , $s_1\geq 1$, $s_2\geq 1$ and $s_1+s_2=p$. It's similar to the subcase 3.1 and 3.2 and we have that 
\begin{equation*}
  \sum_{i=1}^{n} n_{U_{13}} \left(r,\frac{1}{f-b_i}\right)\leq  2n_{U_{13}}(r+1,f)+O(1).  
\end{equation*}
From the above discussion, we have 
\begin{equation*}
  \sum_{i=1}^{n} n \left(r,\frac{1}{f-b_i}\right)=\sum_{i=1}^{n}\sum_{j=11}^{13}n_{U_j}\left(r,\frac{1}{f-b_i}\right)\leq  2n(r+1,f)+O(1).  
\end{equation*}
Therefore, it follows that we get a contradiction by Lemma \ref{lema}. 

\textbf{(5)} Suppose that $f-b_i$ has infinitely many zeros with multiplicity 3 or higher. We will again prove by contradiction. Without loss of generality, let $z_0$ be a generic zero of 
$f-b_1$ with multiplicity $p\geq 3$. From \eqref{eqm1}, $f(z+1)f(z-1)$ has a pole at $z_0$ with multiplicity $p$. 

\textbf{Subcase 3.4} Assume that $U_{14}=\{f(z_0+1)=\frac{h_{22}(z)}{(z-z_0)^{q_5}}, f(z_0-1)\neq \infty\}$ and $z_0\in U_{14}$, where $h_{22}(z)\neq 0, \infty$ in $U_{14}$ and $q_5\geq p$, then $z_0$ is a double pole of $f(z+2)$ and a zero of $f(z+3)$ with multiplicity $q_5-2\geq 1$. 
Similarly to the proof in subcase 3.1, we have 
\begin{equation*}
  \sum_{i=1}^{n} n_{U_{14}} \left(r,\frac{1}{f-b_i}\right)\leq  n_{U_{14}}(r+1,f)+O(1).  
\end{equation*}

\textbf{Subcase 3.5} Suppose that $U_{15}=\{f(z+1)=\frac{h_{23}(z)}{(z-z_0)^{p-1}},f(z-1)=\frac{h_{24}(z)}{z-z_0}\}$ and $z_0\in U_{15}$,  where  $h_m(z)\neq 0, \infty$ in $U_{15}$ for $m\in \{23, 24\}$ , then $f(z_0-2)$ is finite. 
When $p=3$, $z_0+2$ is a double pole of $f$ and $z_0+3$ may be the generic zero of $f-b_i$. When $p\geq 4$, it's similar to the argument in subcase 3.1. So we give 
\begin{equation*}
  \sum_{i=1}^{n} n_{U_{15}} \left(r,\frac{1}{f-b_i}\right)\leq  n_{U_{15}}(r+1,f)+O(1).  
\end{equation*}
 
\textbf{Subcase 3.6} Assume that $U_{16}=\{f(z+1)=\frac{h_{25}(z)}{(z-z_0)^{s_1}}, f(z-1)=\frac{h_{26}(z)}{(z-z_0)^{s_2}}\}$ and $z_0\in U_{16}$, where  $h_m(z)\neq 0, \infty$ in $U_{16}$ for $m\in \{25, 26\}$ and $s_1+s_2=p$. Similarly to the proof in subcase 3.4 and 3.5, we have 
\begin{equation*}
  \sum_{i=1}^{n} n_{U_{16}} \left(r,\frac{1}{f-b_i}\right)\leq  n_{U_{16}}(r+1,f)+O(1).  
\end{equation*}

From the above analysis, we obtain 
\begin{equation*}
  \sum_{i=1}^{n} n \left(r,\frac{1}{f-b_i}\right)=\sum_{i=1}^{n}\sum_{j=14}^{16}n_{U_j}\left(r,\frac{1}{f-b_i}\right)\leq  n(r+1,f)+O(1).  
\end{equation*}
Therefore, we can get $\sigma _f>0$ by Lemma \ref{lem1}, a contradiction. 

\section{Proof of Theorem 3}

\ \ \ \ Notice that $b_1$ is not a solution of \eqref{eqm1}, and $\deg_fQ=2+\deg_f \hat{Q}$. Suppose that $f-b_1$ has only finitely many simple zeros. By the proof in \cite[Lemma 2.3]{Nie-2025}, we have 
\begin{equation*}
  N\left(r,\frac{1}{f-b_1}\right)=T(r,f)+S(r,f).
\end{equation*}
Thus, $f-b_1$ has infinitely many multiple zeros. Let $z_0$ be a generic zero of $f-b_1$ with multiplicity $p(\geq 2)$. 

\textbf{Case 1} Suppose that $\deg_f P=\deg_f Q+1$. If $b_1$ is a constant, then $f'(z_0)=0$ and hence $z_0$ is a double pole of $S(f,z)$. By iteration of shifting, $f(z+1)f(z-1)$ has a pole with multiplicity $2p$ at $z_0$.  

\textbf{Subcase 1.1} Assume that $L_1=\{z:f(z+1)=\frac{l_1(z)}{(z-z_0)^{q_1}},f(z-1)\neq\infty\}$ and $z_0\in L_1$, where $l_1(z)\neq 0, \infty$ in $L_1$ and $q_1\geq 2p$, then $z_0+2$ is a pole of $f$ with multiplicity at least $2p$ and $z_0+3$ may be another generic zero of $f-b_1$. 
Thus we have 
\begin{equation*}
   n_{L_1} \left(r,\frac{1}{f-b_1}\right)\leq  \frac{1}{2}n_{L_1}(r+1,f)+O(1), 
\end{equation*}
where $n_{L_i}(r,f)$ is the number of multiplicities of all poles of $f$ in the set $L_i\bigcap \{z\ :\ |z|<r\}$ for $i\in \mathbb{N} $. The definition of $L_i(i\geq 2)$ will give in the following subcases.

\textbf{Subcase 1.2} Suppose that $L_2=\{z:f(z+1)=\frac{l_2(z)}{(z-z_0)^{2p-1}}, f(z-1)=\frac{l_3(z)}{z-z_0}\}$ and $z_0\in L_2$, where $l_m(z)\neq 0, \infty$ in $L_2$ for $m\in \{2,3\}$, then $z_0-2$ is a simple pole of $f$ and $z_0-3$ is the generic zero of $f-b_1$. Similarly, we find that $z_0+2$ is a pole of $f$ with multiplicity $2p-1$ and $f(z_0+3)$ is finite. 
In summary, we have 
\begin{equation*}
  n_{L_2} \left(r,\frac{1}{f-b_1}\right)\leq  \frac{1}{2}n_{L_2}(r+1,f)+O(1). 
\end{equation*}

\textbf{Subcase 1.3} Suppose that $L_3=\{z:f(z+1)=\frac{l_4(z)}{(z-z_0)^{s_1}}, f(z-1)=\frac{l_5(z)}{(z-z_0)^{s_2}}\}$ and $z_0\in L_3$, where $s_1, s_2\geq 3$, $s_1+s_2=2p$ and $l_m(z)\neq 0, \infty$ in $L_3$ for $m\in \{4,5\}$ . Similar to the argument in subcase 1.1, we find that 
\begin{equation*}
  n_{L_3} \left(r,\frac{1}{f-b_1}\right)\leq  \frac{1}{2}n_{L_3}(r+1,f)+O(1). 
\end{equation*} 
By considering all generic zeros of $f-b_1$ in the set $\bigcup _{i=1}^{3}L_{i}$, it follows that 
\begin{equation*}
  n\left(r,\frac{1}{f-b_1}\right)=\sum_{i=1}^{3}n_{L_i}\left(r,\frac{1}{f-b_1}\right)\leq \frac{1}{2}n(r+1,f)+O(1).
\end{equation*}
Thus $\sigma _f>0$ by Lemma \ref{lema}, a contradiction. Hence, $b_1$ is non-constant. So there exists a point $z_0$ such that $b_1'(z_0)\neq 0$ and $f'(z_0)\neq 0$. Similarly with the case when $b_1$ is constant, we also obtain $\sigma _f>0$, a contradiction.

 \textbf{Case 2} Assume that $\deg_fP=\deg_fQ+2$. If $b_1$ is a constant, then $f'(z_0)=0$ and hence $z_0$ is a double pole of $S(f,z)$. By iteration of shifting, $f(z+1)f(z-1)$ has a pole with multiplicity $2p$ at $z_0$.    

 \textbf{Subcase 2.1} Suppose that $L_4=\{z:f(z+1)=\frac{l_6(z)}{(z-z_0)^{q_2}},f(z-1)\neq\infty\}$ and $z_0\in L_4$, where $l_6(z)\neq 0, \infty$ in $L_4$ and $q_2\geq 2p$. From the recurrence relation, we know that $z_0+d$ is multiple poles of $f$ for any $d\in \mathbb{N} $. Thus, it follows that 
 \begin{equation*}
  n_{L_4} \left(r,\frac{1}{f-b_1}\right)\leq  \frac{1}{2}n_{L_4}(r+1,f)+O(1). 
\end{equation*}

\textbf{Subcase 2.2} Assume that $L_5=\{z:f(z+1)=\frac{l_7(z)}{(z-z_0)^{s_1}}, f(z-1)=\frac{l_{8}(z)}{(z-z_0)^{s_2}}\}$ and $z_0\in L_5$, where $l_m(z)\neq 0, \infty$ in $L_5$ for $m\in \{7,8\}$ , $s_1\geq 1$, $s_2\geq 1$ and $s_1+s_2=2p$. Similar to the proof in subcase 2.1, we have 
\begin{equation*}
  n_{L_5} \left(r,\frac{1}{f-b_1}\right)\leq  \frac{1}{2}n_{L_5}(r+1,f)+O(1). 
\end{equation*}

In summary, it follows that \begin{equation*}
  n\left(r,\frac{1}{f-b_1}\right)=n_{L_4}\left(r,\frac{1}{f-b_1}\right)+n_{L_5} \left(r,\frac{1}{f-b_1}\right)\leq \frac{1}{2}n(r+1,f)+O(1).
\end{equation*}
Thus $\sigma _f>0$ by Lemma \ref{lema}, a contradiction. Hence, $b_1$ is non-constant. So there exists a point $z_0$ such that $b_1'(z_0)\neq 0$ and $f'(z_0)\neq 0$. Similarly with the case when $b_1$ is constant, we also obtain $\sigma _f>0$, a contradiction.

\end{document}